\documentclass[11pt]{amsart}

\usepackage{amsmath,amssymb,mathrsfs}
\usepackage{stmaryrd}
\usepackage[utf8]{inputenc}
\usepackage{tikz-cd}

\begin{document}
\title{The prime spectrum of an $L$-algebra}
\author{Wolfgang Rump and Leandro Vendramin}
\dedicatory{to B.~V.~M.}

\address{Institute for Algebra and
Number Theory, University of Stuttgart, 
Pfaffenwaldring 57, D-70550 Stuttgart, Germany}
\email{rump@mathematik.uni-stuttgart.de}

\address{Department of Mathematics and Data
Science, Vrije Universiteit Brussel, Pleinlaan 2, 1050 Brussel}
\email{Leandro.Vendramin@vub.be}

\keywords{L-algebra, self-similarity}
\subjclass[2010]{06B10, 06F05, 08A55}

\maketitle

\theoremstyle{plain}
\newtheorem{thm}{Theorem}
\newtheorem{cor}[thm]{Corollary}
\newtheorem{prop}{Proposition}
\newtheorem{lem}{Lemma}

\theoremstyle{definition}
\newtheorem{Definition}{Definition}
\newtheorem{exa}{Example}
\newtheorem{rem}{Remark}

\renewcommand{\labelenumi}{\rm (\alph{enumi})}
\newcommand{\hs}{\hspace{2mm}}
\newcommand{\vsp}{\vspace{4ex}}
\newcommand{\vspc}{\vspace{-1ex}}
\newcommand{\hra}{\hookrightarrow}
\newcommand{\tra}{\twoheadrightarrow}
\newcommand{\md}{\mbox{-\bf mod}}
\newcommand{\Mod}{\mbox{-\bf Mod}}
\newcommand{\Mdd}{\mbox{\bf Mod}}
\newcommand{\mdd}{\mbox{\bf mod}}
\newcommand{\latt}{\mbox{-\bf lat}}
\newcommand{\Proj}{\mbox{\bf Proj}}
\newcommand{\Inj}{\mbox{\bf Inj}}
\newcommand{\Ab}{\mbox{\bf Ab}}
\newcommand{\CM}{\mbox{-\bf CM}}
\newcommand{\Prj}{\mbox{-\bf Proj}}
\newcommand{\prj}{\mbox{-\bf proj}}
\newcommand{\ra}{\rightarrow}
\newcommand{\eps}{\varepsilon}
\renewcommand{\epsilon}{\varepsilon}
\renewcommand{\phi}{\varphi}
\renewcommand{\rho}{\varrho}
\renewcommand{\hom}{\mbox{Hom}}
\newcommand{\ex}{\mbox{Ext}}
\newcommand{\rad}{\mbox{Rad}}
\renewcommand{\Im}{\mbox{Im}}
\newcommand{\oti}{\otimes}
\newcommand{\sig}{\sigma}
\newcommand{\en}{\mbox{End}}
\newcommand{\Qq}{\mbox{Q}}
\newcommand{\lra}{\longrightarrow}
\newcommand{\lras}{\mbox{ $\longrightarrow\raisebox{1mm}{\hspace{-6.5mm}$\sim$}\hspace{3mm}$}}
\newcommand{\Eq}{\Leftrightarrow}
\newcommand{\Equ}{\Longleftrightarrow}
\newcommand{\Ra}{\Rightarrow}
\newcommand{\Lra}{\Longrightarrow}
\newcommand{\A}{\mathbb{A}}
\newcommand{\N}{\mathbb{N}}
\newcommand{\Z}{\mathbb{Z}}
\newcommand{\C}{\mathbb{C}}
\newcommand{\Q}{\mathbb{Q}}
\newcommand{\R}{\mathbb{R}}
\renewcommand{\H}{\mathbb{H}}
\newcommand{\K}{\mathbb{K}}
\newcommand{\F}{\mathbb{F}}
\renewcommand{\P}{\mathbb{P}}
\newcommand{\B}{\mathbb{B}}
\newcommand{\rk}{\mbox{r}}
\newcommand{\p}{\mathfrak{p}}
\newcommand{\q}{\mathfrak{q}}
\renewcommand{\k}{\mathfrak{k}}
\newcommand{\AAA}{\mathfrak{A}}
\newcommand{\PPP}{\mathfrak{P}}
\renewcommand{\AA}{\mathscr{A}}
\renewcommand{\SS}{\mathscr{S}}
\newcommand{\DD}{\mathscr{D}}
\newcommand{\BB}{\mathscr{B}}
\newcommand{\LL}{\mathscr{L}}
\newcommand{\HH}{\mathscr{H}}
\newcommand{\NN}{\mathscr{N}}
\newcommand{\CC}{\mathscr{C}}
\newcommand{\EE}{\mathscr{E}}
\newcommand{\MM}{\mathscr{M}}
\newcommand{\OO}{\mathscr{O}}
\newcommand{\PP}{\mathscr{P}}
\newcommand{\II}{\mathscr{I}}
\newcommand{\TT}{\mathscr{T}}
\newcommand{\XX}{\mathscr{X}}
\newcommand{\YY}{\mathscr{Y}}
\newcommand{\FF}{\mathscr{F}}
\newcommand{\GG}{\mathscr{G}}
\newcommand{\RR}{\mathscr{R}}
\newcommand{\setm}{\smallsetminus}
\renewcommand{\le}{\leqslant}
\renewcommand{\ge}{\geqslant}
\newcommand{\rat}{\rightarrowtail}
\newcommand{\op}{^{\mbox{\scriptsize op}}}
\newcommand{\ltra}{\lra\!\!\!\!\!\!\:\ra}
\newcommand{\subs}{\subset}
\newcommand{\sups}{\supset}
\newcommand{\subsn}{\subsetneq}
\newcommand{\nsubs}{\not\subset}
\newcommand{\pt}{\makebox[0pt][r]{\bf .\hspace{.5mm}}}
\newcommand{\dpt}{\makebox[0mm][l]{.}}
\newcommand{\dpc}{\makebox[0mm][l]{,}}
\newcommand{\noth}{\varnothing}
\newcommand{\da}{\downarrow\!\!}
\newcommand{\ua}{\uparrow\!\!}

\begin{abstract}
    We prove that the lattice of ideals of an arbitrary $L$-algebra is distributive. As a consequence,
    a spectral theory applies with no restriction. We also study the spectrum (i.e. the set of prime ideals) 
    of $L$-algebras and characterize prime ideals in topological terms. 
\end{abstract}

\section{Introduction}

The concept of $L$-algebra \cite{Log} represents a quantum structure with a close
relationship to braidings \cite{Gar, DP, D}, non-commutative logic \cite{MT, Ch, BvN, L-Alg},
and the Yang-Baxter equation \cite{ESS, Dual}. For example,
Artin's braid group \cite{Art} and similar torsion-free groups \cite{BS, Del, D, DP,
Deh} are generated by a finite $L$-algebra. More generally, every {\em right $\ell$-group}
\cite{Rum}, that is, a group with a lattice order invariant under right multiplications,
is determined by its negative cone which is an $L$-algebra. So this wide class of
partially ordered groups can be treated by methods from the theory of $L$-algebras. The
projection lattice of a von Neumann algebra is an $L$-algebra, generating a right
$\ell$-group which is a complete invariant \cite{OML}. In functional analysis, $L$-algebras
arise in connection with Riesz spaces \cite{LZ}, a special class of two-sided $\ell$-groups
\cite{BKW, Dar}. 

In \cite{TopL} it is proved that if the lattice $\II(X)$ of ideals of an $L$-algebra $X$ is
distributive, then $\II(X)$ is again an $L$-algebra and even a spatial locale. In other words,
the ideals of $X$ can be identified with the open sets of a topological space $\mbox{Spec}\,X$,
the {\em spectrum} of $X$. For three
classes of $L$-algebras, the distributivity of $\II(X)$ is verified: the HBCK-algebras
\cite{BF93} which became prominent in connection with Wro\'nski's conjecture \cite{Wro,
Kow}, lattice effect algebras \cite{DvPu, Rie2} which formalize unsharp quantum measurement
\cite{GG, Gud}, and sharp discrete $L$-algebras \cite{Rum} which comprise the projective
spaces with a distinguished elliptic polarity \cite{qset}. HBCK-algebras are equivalent to
$CKL$-algebras \cite{Gli}, generalizing Heyting algebras, and {\em MV-algebras} \cite{Ch} which
can be regarded as abelian $\ell$-groups with a distinguished strong order unit \cite{Mun}.     

In this paper, we show that for every $L$-algebra $X$, the lattice of ideals is distributive, so
that the spectral theory of \cite{TopL} applies with no restriction. Our proof is based on
an explicit description of the join $I\vee J$ of two ideals (Proposition~\ref{p2}). We
show that $\mbox{Spec}\,X$ is a sober space with a basis of quasi-compact open sets
(Proposition~\ref{p3}). In general, the intersection of finitely many quasi-compact open
sets is not quasi-compact (Example~\ref{exa:5}), so that $\mbox{Spec}\,X$ need not be a spectral
space, and the quasi-compact open sets need not be an $L$-subalgebra of the locale
$\mbox{Spec}\,X$ (Example~\ref{exa:6}). An ideal $P$ of $X$ is prime if and only if the $L$-algebra
$X/P$ is subdirectly irreducible (Proposition~\ref{p6}). As in commutative ring theory, 
the open and the closed subsets in the spectrum of an $L$-algebras are again spectra of $L$-algebras (Theorem \ref{t4}).

\section{The ideals of an $L$-algebra} 

For a set $X$ with a binary operation $(x,y)\mapsto x\cdot y$, an element $1\in X$ is said to be a
{\em logical unit} \cite{Log} if it satisfies
\begin{equation}
\label{1}
x\cdot x=x\cdot 1=1,\qquad 1\cdot x=x  
\end{equation}
for all $x\in X$. (Note that $x\cdot x=1$ implies that a logical unit is unique.)
If, in addition,
\begin{gather}
(x\cdot y)\cdot(x\cdot z)=(y\cdot x)\cdot(y\cdot z),\label{2}\\
x\cdot y=y\cdot x=1\;\Lra\; x=y \label{3}
\end{gather} 
hold for $x,y,z\in X$, then $X=(X;\cdot)$ is said to be an {\em $L$-algebra} \cite{Log}.

For any $L$-algebra $X$,
\begin{equation}\label{4}
x\le y\; :\Equ\; x\cdot y=1  
\end{equation}
is a partial order. 
The operation of $X$ is sometimes denoted by an arrow $\ra$ since it can be interpreted as
implication
between propositions. Then \eqref{4} is the entailment relation, and 1 characterizes the
true propositions, up to logical equivalence \eqref{3}. Equation \eqref{2} holds in
the most important generalizations of classical logic, including intuitionistic, many-valued,
and quantum logic \cite{L-Alg}.

A subset $I$ of an $L$-algebra $X$ is said to be an {\em ideal} \cite{Log} if $1\in I$ and
\begin{align}
x\in I\text{ and }x\cdot y\in I\; &\Ra\; y\in I, \label{5}\\
x\in I\; &\Ra\; (x\cdot y)\cdot y\in I, \label{6}\\
x\in I\; &\Ra\; y\cdot x\in I\text{ and } y\cdot(x\cdot y)\in I. \label{7}
\end{align}
Every ideal $I$ gives rise to a congruence relation
\begin{equation}\label{8}
x\equiv y\; :\Equ\; x\cdot y\in I\text{ and } y\cdot x\in I  
\end{equation} 
of $X$, that is, $x\equiv y$ implies $z\cdot x\equiv z\cdot y$ and $x\cdot z\equiv y\cdot z$.
We refer to \eqref{8} as congruence {\em modulo $I$}. The equivalence classes with respect to
$\equiv$ form an $L$-algebra $X/I$ with the induced $L$-algebra operation. Conversely, every
congruence $\equiv$ 
comes from an ideal $I:=\{x\in X\:|\:x\equiv 1\}$. For example, every {\em morphism} of
$L$-algebras, that is, a map $f\colon X\ra Y$ with $f(x\cdot y)=f(x)\cdot f(y)$ for all
$x,y\in X$, determines an ideal $\mbox{Ker}(f):=\{x\in X\:|\:f(x)=1\}$ of $X$, the
{\em kernel} of $f$, and $X/\mbox{Ker}(f)$ is isomorphic to the {\em image} $\mbox{Im}(f):=
f(X)$ of $f$. This is precisely the property of normality of  
the pointed category of $L$-algebras; see of \cite{Ja}.
A subset $Y$ of an $L$-algebra $X$ is called an {\em $L$-subalgebra} if $Y$
is closed with respect to the operation of $X$. Thus, with the induced operation, $Y$ is
an $L$-algebra with a morphism $Y\hra X$. Dually, the ideals of $X$ correspond to the
surjective morphisms $X\tra Z$.  

Thus, from a categorical point of view, ideals of an $L$-algebra behave quite similar to
ideals of a ring or normal subgroups of a group, despite their more involved definition. 
This comes from the fact that $L$-algebras form an ideal determined category \cite{JMTU}.

Recall that an $L$-algebra $X$ is said to be {\em self-similar} if
for each $y\in X$, the downset $\da y:=\{z\in X\:|\:z\le y\}$ is mapped bijectively onto $X$
under the map $\sigma_y\colon\da y\ra X$ with $\sigma_y(z):=y\cdot z$. This provides $X$
with a multiplication $xy:=\sigma_y^{-1}(x)$ which makes $X$ into a monoid with unit element
1. In general, the maps $\sigma_y$ are injective, which leads to a partial multiplication
for any $L$-algebra $X$: For $x,y\in X$, the product $xy$ exists and is equal to $z$ if and
only if $y\cdot z=x$ and $z\le y$.  

By \cite{Log}, Theorem~1 and Proposition~5, a self-similar $L$-algebra can be
characterized as a monoid $A$ (with juxtaposition as operation and unit element 1) with a
second binary operation $A\times A\to A$, $(a,b)\mapsto a\cdot b$ such that the equations   
\begin{align}
a\cdot ba &= b,\label{9} \\
ab\cdot c &= a\cdot(b\cdot c),\label{10}\\
(a\cdot b)a &= (b\cdot a)b,\label{11}
\end{align}
hold in $A$. So the category of self-similar $L$-algebras is a 
variety in the sense of universal algebra. 
It is easily checked that $(A;\cdot)$ is indeed an $L$-algebra. For example,
Eq.~\eqref{2} follows by Eqs.~\eqref{10} and \eqref{11}. Note that by \eqref{4}, Eq.~\eqref{10} 
yields
\begin{equation}\label{12}
ab\le c\;\Equ\; a\le b\cdot c,  
\end{equation} 
which shows that the monoid operation of a self-similar $L$-algebra $A$ is uniquely determined
by the $L$-algebra structure. By \cite{Log}, Theorem~3, the self-similar closure 
is characterized by a property which can easily be checked in concrete examples:

\begin{thm}
\label{t1}
Let $X$ be an $L$-subalgebra of a self-similar $L$-algebra $A$. Then $A\cong S(X)$ if
and only if the monoid $A$ is generated by $X$.
\end{thm} 

In particular, $X$ is self-similar if and only if $X\cong S(X)$.

\begin{exa}
\label{exa:1}
Let $G$ be a {\em right $\ell$-group} \cite{Rum}, that is, a group
with a lattice order such that $a\le b\Ra ac\le bc$ holds for all $a,b,c\in G$. In other words,
the lattice operations are invariant under right multiplication. The {\em negative cone}
$G^-:=\{a\in G\:|\:a\le 1\}$ is a self-similar $L$-algebra with
$$a\cdot b:=ba^{-1}\wedge 1.$$
\end{exa}

For special cases and other examples, see \cite{Log, Gli, Rum, L-Alg}. Here we mention two
special cases which may give a rough idea of the self-similar closure. Firstly, let
$\CC(K)$ be the (two-sided) $\ell$-group of continuous real functions on a compact
space $K$, with the pointwise lattice structure. The functions $f$ with $-1\le f\le 0$ form an
$L$-subalgebra $X$ of the negative cone $\CC(K)^-$. By Theorem~\ref{t1}, $\CC(K)^-$ is the
self-similar closure of $X$. If $K$ is a singleton, $\CC(K)\cong(\R;+)$, and $X=[-1,0]$,
with logical unit 0.

\begin{exa}
\label{exa:2}
The category of $L$-algebras has a generator, the two-element
Boolean algebra $\B=\{0,1\}$. Its self-similar closure is the negative cone of the infinite
cyclic group $\langle 0\rangle$ with its natural order $0^n\le 0^m\Eq n\ge m$.
\end{exa}

By \cite{Log}, Proposition~4, every self-similar $L$-algebra $A$ is a $\wedge$-semilattice with
\begin{equation}\label{13}
a\wedge b:=(a\cdot b)a=(b\cdot a)b.  
\end{equation} 

\begin{exa}
A semilattice $X$ with a greatest element 1 and a
binary operation $(X;\cdot)$ is said to be a {\em Brouwerian semilattice} \cite{Koh}
if it satisfies
\begin{equation}\label{14}
x\wedge y\le z\;\Equ\;x\le y\cdot z.  
\end{equation} 
For a Brouwerian semilattice $X$, Eqs.~\eqref{10} and \eqref{11} give 
\[
(x\wedge y)\cdot z=
(x\cdot y)x\cdot z=(x\cdot y)\cdot(x\cdot z),
\]
which yields Eq.~\eqref{2}. Furthermore,
$x\le y\Eq 1\wedge x\le y\Eq 1\le x\cdot y$ gives \eqref{4}, and thus \eqref{3}, and it is easily
checked that 1 is a logical unit. Thus $X$ is an $L$-algebra. As an $L$-subalgebra of $S(X)$,
a Brouwerian semilattice $X$ satisfies $xy\le x\wedge y$ for all $x,y\in X$, but this 
inequality is not an equality if $|X|\ne 1$.
\end{exa}

A Brouwerian semilattice with a complete lattice order is said to be a {\em locale}
\cite{Joh, MM}. Equivalently, a locale is a complete lattice $A$ satisfying 
$$a\wedge\bigvee_{i\in I}a_i=\bigvee_{i\in I}(a\wedge a_i).$$
The $L$-algebra operation is then given by
\[
a\cdot b:=\bigvee\{c\in A\:|\:c\wedge a\le b\}.
\]

An element $p<1$ of an $L$-algebra $X$ is said to be {\em prime} \cite{Dual, L-Alg} if  
for all $x\in X$, either $x\le p$ or $x\cdot p\le p$. The set of prime elements of $X$ is
denoted by $P(X)$. For a Brouwerian semilattice $X$, an element $p<1$ is prime if and only if 
\begin{equation}\label{15}
x\wedge y\le p\;\Lra\; (x\le p\text{ or } y\le p)  
\end{equation}
holds in $X$. Indeed, if $p$ satisfies \eqref{15}, then \eqref{14} implies that
$(x\cdot p)\wedge x\le p$, which yields $x\cdot p\le p$ or $x\le p$. Conversely, if
$p\in P(X)$, then $x\wedge y\le p$ and $y\not\le p$ implies that $x\le y\cdot p\le p$.

\begin{Definition}
We call an $L$-algebra $X$ {\em spatial} if 
\[
x=\bigwedge\{p\in P(X)\:|\:x\le p\}
\]
holds
for all $x\in X$.  
\end{Definition}

Thus, a locale is spatial \cite{Joh} if and only if it is spatial as an $L$-algebra.    
For an $L$-algebra $X$, the partially ordered set $\II(X)$ of ideals is closed with
respect to intersections $\bigcap S$ of subsets $S\subs\II(X)$. Hence $\II(X)$
is a complete lattice.

\section{The spectrum of an $L$-algebra} 

In this section, we prove that the lattice of ideals of an $L$-algebra $X$ is distributive,
which settles an open problem in \cite{TopL}. As a consequence, this will imply that $\II(X)$
is a spatial locale. Our proof makes use of the following correspondence between $\II(X)$
and $\II(S(X))$ (see the corollaries of \cite{Gli}, Theorem~1):

\begin{thm}
\label{t2}
Let $X$ be an $L$-algebra. The maps $I\mapsto S(I)$ and $J\cap X\mapsfrom J$ establish a
bijective correspondence between the ideals $I$ of $X$ and the ideals $J$ of $S(X)$,
where the self-similar closure $S(I)$ is equal to the ideal of $S(X)$ generated by $I$.
\end{thm}

In what follows, we
write $x\equiv y$ (mod $I$) for the congruence modulo an $L$-algebra
ideal $I$. For ideals $I$ and $J$ of an $L$-algebra $X$, we write $I\vee J$
for the join in $\II(X)$.

\begin{cor}
Let $I$ be an ideal of an $L$-algebra $X$, and $x,y\in X$.
Then $x\equiv y$ \rm (mod $I$) \it if and only if $x\equiv y$ \rm (mod $S(I)$).
\end{cor}

\begin{proof}
It follows immediately from Theorem \ref{t2}.
\end{proof}

\begin{prop}
\label{p1}
Each congruence of a self-similar $L$-algebra $A$ is a congruence of $A$ as a monoid.
\end{prop}

\begin{proof}
Let $\equiv$ be a congruence of $A$ as an $L$-algebra, and let $I$ be the corresponding
ideal. Assume that $a\equiv b$. For any $c\in A$, Eqs.~\eqref{9} and \eqref{10} give
$ac\cdot bc=a\cdot(c\cdot bc)=a\cdot b\in I$. By symmetry, this implies that $ac\equiv bc$.
Similarly, $ca\cdot cb=c\cdot(a\cdot cb)\equiv c\cdot(b\cdot cb)=c\cdot c=1$, which yields
$ca\cdot cb\in I$. By symmetry, $ca\equiv cb$. 
\end{proof}

The converse of Proposition~\ref{p1} is not true. 

\begin{exa}
Let $Z:=\{x^n\:|\:n\in\Z\}$ be the infinite cyclic group with the partial order 
\[
x^n\le x^m
\Eq n\ge m.
\]
Its negative cone $Z^-$ is a self-similar $L$-algebra (Example~\ref{exa:2}). The
subsets $\{1\}$ and $Z^-\setm\{1\}$ are the equivalence classes of a congruence relation
of the monoid $Z^-$. However, $x\cdot x=1\not\equiv x=x\cdot x^2$, which shows that $\equiv$
is not a congruence for the $L$-algebra $Z^-$.
\end{exa}

\begin{prop}
\label{p2}
Let $I$ and $J$ be ideals of an $L$-algebra $X$. Then $y\in X$ belongs to $I\vee J$ if and
only if there is an element $x\in I$ with $x\equiv y$ \rm (mod $J$). \it     
\end{prop}

\begin{proof}
Note first that $I$ and $J$ are contained in the set $V$ of all $y\in X$ which admit an 
element $x\in I$ with $x\equiv y$ (mod $J$). On the other hand, \eqref{5} implies that $V
\subs I\vee J$. So we only have to show that $V$ is an ideal. To verify \eqref{5}, let
$x,y\in X$ be elements with $x\in V$ and $x\cdot y\in V$. So there exist $z,t\in I$ with $z\equiv
x$ (mod $J$) and $t\equiv x\cdot y$ (mod $J$). By Proposition~\ref{p1}, this implies that
$tz\equiv (x\cdot y)x$ (mod $S(J)$). Hence $tz\cdot(x\cdot y)x\in S(J)$. By \eqref{12}, 
$(x\cdot y)x\le y$. Thus Eqs.~\eqref{13} yield $\bigl(tz\cdot(x\cdot y)x\bigr)tz=
tz\wedge(x\cdot y)x\le y$. So we obtain $tz\cdot(x\cdot y)x\le tz\cdot y$, which gives
$tz\cdot y\in S(J)$. By Eq.~\eqref{10}, it follows that $tz\cdot y=t\cdot(z\cdot y)\in
S(J)\cap X=J$.
By \eqref{6} and \eqref{7}, this yields $(tz\cdot y)\cdot y\equiv y$ (mod $J$). Since
$(tz\cdot y)\cdot y\in S(I)\cap X=I$, it follows  that $y\in V$. Thus $V$
satisfies \eqref{5}.  

Next assume that $x\in V$ and $y\in X$. So there exists an element $z\in I$ with $z\equiv 
x$ \mbox{(mod $J$)}. Then $(x\cdot y)\cdot y\equiv(z\cdot y)\cdot y$ \mbox{(mod $J$)} and 
$(z\cdot y)\cdot y\in I$. So $(x\cdot y)\cdot y\in V$. This shows that $V$ satisfies 
\eqref{6}. Similarly, $y\cdot x\equiv y\cdot z$ \mbox{(mod $J$)} and 
$y\cdot(x\cdot y)\equiv y\cdot(z\cdot y)$ \mbox{(mod $J$)}. Since $y\cdot z\in I$ and
$y\cdot(z\cdot y)\in I$, this completes the proof that $V$ is an ideal.
\end{proof}

In particular, the characterization of $I\vee J$ in Proposition~\ref{p2} is symmetric
with respect to $I$ and $J$. Now we are ready to prove our main result.

\begin{thm}
\label{t3}
The lattice of ideals of an $L$-algebra $X$ is distributive. 
\end{thm} 

\begin{proof}
Let $I,J,K$ be ideals of $X$. To show that $(I\vee J)\cap K=(I\cap K)\vee(J\cap K)$,
it suffices to verify the inclusion $(I\vee J)\cap K\subs(I\cap K)\vee(J\cap K)$. Let $z\in
(I\vee J)\cap K$ be given. By Proposition~\ref{p2}, there exists an element $x\in I$ with
$x\equiv z$ (mod $J$). Hence $x\cdot z\in J\cap K$. Since $x\in I$, \eqref{6} and \eqref{7}
imply that $x\cdot z\equiv z$ (mod $I$). Thus, $x\cdot z\in K$ and $z\in K$ yields
$x\cdot z\equiv z$ (mod $I\cap K$). By Proposition~\ref{p2}, this proves that $z\in
(I\cap K)\vee(J\cap K)$. 
\end{proof}

\begin{cor}
For any $L$-algebra $X$, the lattice $\II(X)$ of ideals is
a spatial locale. The $L$-algebra operation of $\II(X)$ is given by
$$I\cdot J=\{x\in X\:|\:\langle x\rangle\cap I\subs J\},$$ 
where $\langle x\rangle$ denotes the ideal generated by $x$. \rm
\end{cor}

\begin{proof}
It follows immediately from 
Theorem~\ref{t2}.
\end{proof}

Equivalently, $I\cdot J$ is the greatest ideal $K$ with $K\cap I\subs J$. 

An ideal $P$ is prime if for every ideal $I$ 
either $I\subseteq P$ or $I\cdot P\subseteq P$.
The prime ideals
of $\II(X)$ form a topological space $\mbox{Spec}\,X:=P(\II(X))$, the {\em spectrum} of $X$,
such that the map $I\mapsto U_I$ with 
$$U_I:=\{P\in\mbox{Spec}\,X\:|\:I\nsubs P\}$$
gives a bijection from $\II(X)$ onto the open sets of $\mbox{Spec}\,X$ (see \cite{TopL},
corollary of Proposition~10). Recall that a topological space is said to be {\em sober}
\cite{MM} if every non-empty closed irreducible set $A$ has a unique {\em generic point} $x$, 
that is, $A$ is the closure of $\{x\}$. 

\begin{prop}
\label{p3}
Let $X$ be an $L$-algebra. The spectrum of $X$ is a sober $T_0$-space
such that the set \rm
$\DD(\mbox{Spec}\,X)$ \it of quasicompact open sets is a basis of \rm $\mbox{Spec}\,X$. 
\end{prop} 

\begin{proof} For distinct points $P,Q$ with $Q\nsubs P$, we have $P\in U_Q$ and $Q\nsubs U_Q$. Hence
$\mbox{Spec}\,X$ is a $T_0$-space. An open set $U_I$ is quasicompact if and only if the
corresponding ideal $I$ is finitely generated. Thus $\DD(\mbox{Spec}\,X)$ is a basis of
$\mbox{Spec}\,X$. To verify that $\mbox{Spec}\,X$ is sober, let $A$ be a non-empty closed
irreducible set in $\mbox{Spec}\,X$. So
there is an ideal $P$ with $A=\mbox{Spec}\,X\setm U_P$. The irreducibility of $A$ says that
$A\subs B\cup C$ with closed sets $B$ and $C$ implies that $A\subs B$ or $A\subs C$. Thus,
if $I,J$ are ideals, then $U_I\cap U_J\subs U_P$ implies that $U_I\subs U_P$ or $U_J\subs
U_P$. Equivalently, $I\cap J\subs P$ implies that $I\subs P$ or $J\subs P$. Since $\II(X)$
is a Brouwerian semilattice, \eqref{15} shows that $P$ is a prime ideal. Thus $A$ is the
closure of $\{P\}$. \end{proof}

\begin{rem}
A topological space with the conditions of Proposition~\ref{p3} and
the additional property that the intersection of quasi-compact open sets is quasi-compact
is said to be a {\em generalized spectral space} \cite{Sto, CGL}.
Hochster \cite{Ho} has shown that these spaces coincide with the topological spaces
underlying a separated scheme. The following example shows that the spectrum of an $L$-algebra
need not be of that type.
\end{rem}

Let $X$ and $Y$ be $L$-algebras. We define the {\em ordered sum}
$X\olessthan Y$ to be the set $(X\setm\{1\})\sqcup Y$ with the induced $L$-algebra operations
of $X$ and $Y$ together with $x\cdot y=1$ and $y\cdot x=x$ for $x\in X\setm\{1\}$ and $y\in
Y$. It is easily checked that $X\olessthan Y$ is an $L$-algebra with $x<y$ for $x\in X
\setm\{1\}$ and $y\in Y$. 

\begin{exa}
\label{exa:5}
Let $Y$ be any $L$-algebra that is not finitely generated and let $X:=\{1,p,q\}$ where $p$ and $q$ 
are incomparable
prime elements. The ideals $P:=Y\cup\{p\}$ and $Q:=Y\cup\{q\}$ of $X\olessthan Y$ are
finitely generated, but $P\cap Q=Y$ is not finitely generated. Thus
$\DD(\mbox{Spec}\,X\olessthan Y)$ is not closed with respect to intersection.
\end{exa}

For an $L$-algebra $X$, the following example shows that the finitely generated ideals need not 
be an $L$-subalgebra of $\II(X)$.

Let $\Omega$ be a partially ordered set with a greatest element 1.
Define $x\cdot y:=1$ if $x\le y$ and $x\cdot y:=y$ otherwise. This makes $\Omega$ into an
$L$-algebra (see \cite{Log}, Example~\ref{exa:1}). The ideals of $\Omega$ are the upper sets.

\begin{exa}
\label{exa:6}
Let $\Omega:=\{y,1,x,x_2,x_3,\ldots\}$ with $1>x>x_2>x_3>\cdots$ and $y<x$.
The ideals $I:=\{y,x,1\}$ and $J:=\{x,1\}$ are finitely generated, but $I\cdot J=
\{1,x,x_2,x_3,\ldots\}$ is not finitely generated. 
\end{exa}

The relationship between prime elements and prime ideals of an $L$-algebra depends on
the following 

\begin{Definition}
Let $X$ be an $L$-algebra. We call an element $q<1$ in $X$ {\em quasi-prime} if
the implication $q\in I\vee J\Ra q\in I\cup J$ holds for any pair of ideals $I,J\in\II(X)$.
The set of quasi-prime elements of $X$ will be denoted by $QP(X)$.
\end{Definition} 

For any $x\in X$, Zorn's lemma shows the existence of ideals $P$ which are maximal among the
ideals $I$ with $x\notin I$. Such an ideal $P$ must be prime (see the proof of
\cite{TopL}, Proposition~10). If $x$ is quasi-prime, the ideal $P_x:=P$ is unique. So
$x\mapsto P_x$ gives a natural map
\begin{equation}
\label{16}
QP(X)\lra\mbox{Spec}\,X.  
\end{equation} 

\begin{prop}
    Every prime element of an $L$-algebra $X$ is quasi-prime.     
\end{prop} 

\begin{proof} 
Let $p\in X$ be prime, and let $I,J\in\II(X)$ be ideals with $p\in I\vee J$. By
Proposition~\ref{p2}, there is an element $x\in I$ with $x\equiv p$ (mod $J$). If $x\le p$,
this implies that $p\in I$. Otherwise, $x\cdot p\le p$, which yields $p\in J$. 
\end{proof}

Thus every prime element $p\in P(X)$ determines a prime ideal
$P_p$. The map (\ref{16}) need not be injective:

\begin{exa}
Let $X=\{1,p,q,0\}$ be the $L$-algebra with underlying partial
order
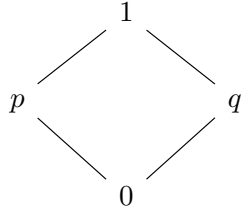
\[
\begin{tikzcd}
	& 1 \\
	p && q \\
	& 0
	\arrow[no head, from=1-2, to=2-1]
	\arrow[no head, from=2-1, to=3-2]
	\arrow[no head, from=3-2, to=2-3]
	\arrow[no head, from=2-3, to=1-2]
\end{tikzcd}
\]
and prime elements $p$ and $q$ satisfying $p\cdot q=p\cdot 0=q\cdot p=q\cdot 0=0$. Hence 
$\langle p\rangle=\langle q\rangle=X$, and thus $P_p=P_q=\{1\}$.
\end{exa}

\begin{rem}
Quasi-prime elements need not be prime. Recall that an $L$-algebra
$X$ is said to be {\em sharp} \cite{qset} if it satisfies the equation
$$x\cdot(x\cdot y)=x\cdot y;$$
see \cite{GG, Gud} for its quantum mechanical meaning.
If all elements $x<1$ 
are maximal, the $L$-algebra is said to be {\em discrete} \cite{Rum}. By \cite{Rum},
Proposition~18, there is a one-to-one correspondence between discrete $L$-algebras and a
class of geometric lattices \cite{Grae}. The proof of \cite{TopL}, Theorem~3, shows that
each element $x<1$ of a sharp discrete $L$-algebra $X$ is quasi-prime. On the other hand, an
element $p<1$ of $X$ is prime if and only if $q\cdot p=p$ for all $q\ne p$. So there need
not exist any prime element in $X$.
\end{rem}

\section{The spectrum of special $L$-algebras} 

A product in the category of $L$-algebras is given by the cartesian product $\prod X_i$ with
component-wise operation. There are natural embeddings 
\[
X_j\hra\prod X_i 
\]
which map $x\in X_j$ 
to the element with $j$-th component $x$ and all other components 1. Since $X_j$ is the
kernel of the canonical map $\prod X_i\tra\prod_{i\ne j}X_i$, the $X_j$ are ideals of $\prod X_i$. For 
$x\in X_j$ and $y\in X_k$ with distinct $j,k$, we have $x\cdot y=y$ in $\prod X_i$.

\begin{prop}
Let $X$ and $Y$ be $L$-algebras. Then $S(X\times Y)=S(X)\times S(Y)$. In particular, each
element of $X\times Y$ has a unique representation 
\[
xy=yx=x\wedge y
\]
with $x\in X$ and $y\in Y$.
\end{prop} 

\begin{proof} 
Since $S(X)\times S(Y)$ is self-similar, and $X\times Y$ generates the monoid
$S(X)\times S(Y)$, Theorem~\ref{t1} implies that $S(X\times Y)=S(X)\times S(Y)$. For $(x,y)
\in X\times Y$, 
\[
(x,y)=(x,1)(1,y)=(1\cdot x,y\cdot 1)(1,y)=(x,1)
\wedge(1,y)
\]
by Eqs.~\eqref{13}.
\end{proof} 

As in 
universal algebra, we define

\begin{Definition}
We say that an $L$-algebra $X$ is a {\em subdirect product} of $L$-algebras $Y$ and $Z$
if $X$ is an $L$-subalgebra of $Y\times Z$ such that the projections of $Y\times Z$ map $X$
onto the factors $Y$ and $Z$. If $X$ does not admit a subdirect product with non-invertible
projections $Y\twoheadleftarrow X\tra Z$, we call $X$ {\em subdirectly irreducible}.
\end{Definition} 

\begin{prop}
\label{p6}
An ideal $P$ of an $L$-algebra $X$ is prime if and only if $X/P$ is subdirectly irreducible.
\end{prop} 

\begin{proof} Let $P$ be prime, and let $X/P\hra Y\times Z$ be a subdirect product. So there is a
morphism $f\colon X\ra Y\times Z$ with kernel $P$. Let $p\colon Y\times Z\tra Y$ and
$q\colon Y\times Z\tra Z$ be the canonical projections. Then 
\[
P=\mbox{Ker}(pf)\cap
\mbox{Ker}(qf).
\]
Since $P$ is prime, $\mbox{Ker}(pf)=P$ or $\mbox{Ker}(qf)=P$. Thus $X/P$
is subdirectly irreducible. 

Conversely, let $X/P$ be subdirectly irreducible. For ideals
$I,J\in\II(X)$ with $I\cap J=P$ this implies that the subdirect product $X/P\hra X/I\times
X/J$ is trivial, that is, $I=P$ or $J=P$. More generally, let $I,J$ be ideals with $I\cap
J\subs P$. Then $(I\vee P)\cap(J\vee P)=(I\cap J)\vee P=P$. Hence $I\vee P=P$ or $J\vee P
=P$. By \eqref{15}, this proves that $P$ is a prime ideal. 
\end{proof}  

\begin{prop}
Let $X$ and $Y$ be $L$-algebras. Then 
\[
\II(X\times Y)=\II(X)\times\II(Y)\text{ and }\mbox{Spec}(X\times Y)=\mbox{Spec}\,X\sqcup\mbox{Spec}\,Y.
\]
\end{prop}

\begin{proof} 
For $I\in\II(X\times Y)$, Theorem~\ref{t3} gives 
\[
I=(I\cap X)\vee(I\cap Y)=
(I\cap X)\times(I\cap Y).
\]
Since $I\cap X$ and $I\cap Y$ are ideals of $X$ and $Y$, respectively,
this proves the first statement. For a prime ideal $P$ of $X\times Y$, Proposition~\ref{p6}
implies that $P\cap X=X$ or $P\cap Y=Y$. So we can assume without loss of generality that
$X\subs P$. Thus $P=X\times J$ with a prime ideal $J$ of $Y$. Conversely, every ideal
$X\times J$ with $J\in\mbox{Spec}\,Y$ is a prime ideal of $X\times Y$. 
\end{proof}

The next result shows that, as in commutative ring theory, under some
assumptions, the open and the closed subsets in the spectrum of an $L$-algebra
are again spectra of $L$-algebras.

\begin{thm}
\label{t4}
Let $X$ be an $L$-algebra such that $x\cdot (y\cdot z)=y\cdot (x\cdot z)$ for
all $x,y,z\in X$, and let $I$ be an ideal of $X$. Then $P\mapsto I\cdot P$ and
$Q\cap I\mapsfrom Q$ give a bijective correspondence between the prime ideals
$P$ of $I$ and the prime ideals $Q$ of $X$ in the open set $U_I$. 
\end{thm} 

\begin{proof} 
Let $P\in\II(I)$ be given. To show that $P\in\II(X)$, let $x,y\in X$ be elements with 
$x\in P$ and $x\cdot y\in P$. Since $I$ is an ideal, this implies that $y\in I$. Hence
$y\in P$. Since $\II(X)$ is a Brouwerian semilattice, it follows that $P\in\II(X)$. Thus
$I\cdot P\in\II(X)$, and $(I\cdot P)\cap I=P$. Now assume that $P$ is prime, and let $J,K$
be ideals of $X$ with $J\cap K\subs I\cdot P$. Then $J\cap K\cap I\subs P$, which yields
$J\cap I\subs P$ or $K\cap I\subs P$. Hence $J\subs I\cdot P$ or $K\subs I\cdot P$, and thus
$I\cdot P$ is a prime ideal of $X$. Moreover, $I\subs I\cdot P$ would imply that $I=I\cap I
\subs P$, a contradiction. Thus $I\cdot P\in U_I$.

Conversely, a similar argument shows that any prime ideal $Q\in U_I$ of $X$ 
intersects into a prime ideal of $I$. Since $I\nsubs Q$ and $Q$ is prime, $I\cdot Q\subs Q$.
Thus $I\cdot Q=Q$. 
\end{proof} 

As a consequence, under the assumptions of Theorem~\ref{t4}, 
one finds that every open set $U_I$ of $X$ and its 
complement $A_I:=\mbox{Spec}\,X\setm U_I$
are spectra of $L$-algebras with respect to the induced topology:
$$\mbox{Spec}\,I\cong U_I,\qquad\mbox{Spec}\,X/I\cong A_I.$$
Indeed, $U_I\cap U_J=U_{I\cap J}$ shows that $P\mapsto I\cdot P$ identifies $\mbox{Spec}\,I$
with the open subspace $U_I$ of $\mbox{Spec}\,X$. 
For any $L$-algebra $X$ there is monotone map
\begin{equation}\label{17}
i\colon X\op\lra\DD(\mbox{Spec}\,X)   
\end{equation}
with $i(x):=U_{\langle x\rangle}$. Let $C(X)$ be the {\em $\wedge$-closure} of $X$, that is,
the $L$-subalgebra of all finite meets $x_1\wedge\cdots\wedge x_n\in S(X)$ with $x_1,\ldots,
x_n\in X$. By Eqs.~(\ref{10}) and (\ref{13}), and \cite{Log}, Proposition~4, the equations
\begin{align}
(a\wedge b)\cdot c &= (a\cdot b)\:\cdot\:(a\cdot c)\label{18}\\
a\cdot(b\wedge c) &= (a\cdot b)\wedge(a\cdot c).\label{19}
\end{align}
hold in $C(X)$. 

\begin{prop}
\label{p8}
Let $X$ be a $\wedge$-closed $L$-algebra. Then the map $(\ref{17})$ is surjective.     
\end{prop}

\begin{proof}
For $x,y\in X$, we have
\begin{equation}\label{20}
\langle x\rangle\vee\langle y\rangle=\langle x\wedge y\rangle.  
\end{equation}
Indeed, if $x,y\in I$ for some ideal $I$ of $X$, then Eq.~(\ref{19}) and (\ref{7}) give
\[
x\cdot(x\wedge y)=x\cdot y\in I.
\]
Thus (\ref{5}) yields $x\wedge y\in I$. Conversely,
$x\wedge y\in I$ implies that $x,y\in I$ since $I$ is an upper set. By induction, every
finitely generated ideal of $X$ is principal. Whence (\ref{17}) is a surjective map. 
\end{proof}

\begin{prop}
For a $\wedge$-closed $L$-algebra $X$, the map $(\ref{17})$ is bijective if and only if $X$
is a Brouwerian semilattice.
\end{prop} 

\begin{proof}
Assume that the map (\ref{17}) is injective. For $x,y\in X$, (\ref{5}) and (\ref{7}) give
$$\langle x\rangle\vee\langle y\rangle=\langle x\rangle\vee\langle x\cdot y\rangle.$$
Hence Eq.~(\ref{20}) yields $\langle x\wedge y\rangle=\langle x\wedge(x\cdot y)\rangle$.
Thus $x\wedge y=x\wedge(x\cdot y)$. Similarly,
$\langle y\rangle\vee\langle x\cdot y\rangle=\langle y\rangle$ yields $y\wedge
(x\cdot y)=y$, that is, $y\le x\cdot y$. To verify (\ref{14}), assume that $x\wedge y\le z$.
By Eqs.~(\ref{13}), this implies that $(y\cdot x)y\le z$. So Eq.~(\ref{10}) gives
$x\le y\cdot x\le y\cdot z$. Conversely, $x\le y\cdot z$ implies that $x\wedge y\le y
\wedge(y\cdot z)=y\wedge z\le z$. Thus, $X$ is a Brouwerian semilattice. 

Conversely, let $X$ be a Brouwerian semilattice. We show that the ideal $\langle x\rangle$
generated by $x\in X$ is the upper set $\ua x:=\{y\in X\:|\:y\ge x\}$. To verify (\ref{5}),
assume that $y\in\ua x$ and $y\cdot z\in\ua x$. Then $x=x\wedge y\le z$, which yields $z\in
\ua x$. For $x,y\in X$, the implications $x\cdot y\le x\cdot y\Ra (x\cdot y)\wedge x\le y\Ra
x\le(x\cdot y)\cdot y$ show that $\ua x$ satisfies (\ref{6}). Furthermore, $x\wedge y\le x$
yields $x\le y\cdot x$, and $x\wedge y\wedge x\le y$ implies that $x\le y\cdot(x\cdot y)$.
Hence $\ua x$ is an ideal, and thus $\ua x=\langle x\rangle$. So the map (\ref{17}) is
injective. By Proposition~\ref{p8}, it is surjective. 
\end{proof}






\section*{Acknowledgments}

Leandro Vendramin was supported in part by OZR3762 
of Vrije Universiteit Brussel
and FWO Senior Research Project G004124N.
We express our gratitude to Carsten Dietzel, Silvia Properzi and
Yufei Qin for comments and corrections, and  
to the referee for their thorough reading and valuable suggestions.

\end{document}